\documentclass[a4paper, 11pt]{amsart}
\title[BG inequality for hypersurfaces]
{On the Bogomolov-Gieseker inequality for hypersurfaces in the projective spaces}
\date{}
\author{Naoki Koseki}

\makeatletter
 
  \@addtoreset{equation}{section}
 \makeatother

\usepackage{amsmath, amssymb, amsthm, amscd, comment, mathtools, color}
\usepackage[frame,cmtip,curve,arrow,matrix,line,graph]{xy}
\usepackage{tikz}
\usetikzlibrary{intersections, calc}

\theoremstyle{plain}
\newtheorem{thm}{Theorem}[section]
\newtheorem{prop}[thm]{Proposition}
\newtheorem{lem}[thm]{Lemma}
\newtheorem{cor}[thm]{Corollary}
\newtheorem*{thm*}{Theorem}

\theoremstyle{definition}
\newtheorem{defin}[thm]{Definition}

\newtheorem*{NaC}{Notation and Convention}
\newtheorem*{ACK}{Acknowledgement}

\theoremstyle{remark}
\newtheorem{rmk}[thm]{Remark}

\DeclareMathOperator{\ch}{ch}
\DeclareMathOperator{\td}{td}

\newcommand{\bP}{\mathbb{P}}

\newcommand{\bR}{\mathbb{R}}
\newcommand{\bZ}{\mathbb{Z}}

\newcommand{\mcD}{\mathcal{D}}

\newcommand{\mcF}{\mathcal{F}}

\newcommand{\mcH}{\mathcal{H}}

\newcommand{\mcO}{\mathcal{O}}

\newcommand{\mcS}{\mathcal{S}}
\newcommand{\mcT}{\mathcal{T}}

\DeclareMathOperator{\Coh}{Coh}

\DeclareMathOperator{\ext}{ext}

\DeclareMathOperator{\cha}{char}

\begin{document}
\maketitle

\begin{abstract}
We investigate 
the stronger form of the Bogomolov-Gieseker inequality 
on smooth hypersurfaces in the projective space 
of any degree and dimension. 
The main technical tool is the theory of tilt-stability conditions 
in the derived category. 
\end{abstract}

\setcounter{tocdepth}{1}
\tableofcontents

\section{Introduction}
\subsection{Motivation and results}
One of the most important theorems 
in the study of vector bundles on algebraic varieties 
in characteristic zero 
is the following Bogomolov-Gieseker (BG) inequality \cite{bog78, gie79}: 
\begin{equation} \label{eq:bgintro}
\frac{H^{n-2}\ch_2(E)}{H^n\ch_0(E)} \leq 
\frac{1}{2}\left(\frac{H^{n-1}\ch_1(E)}{H^n\ch_0(E)} \right)^2, 
\end{equation}
where $E$ is a slope semistable vector bundle 
on a polarized smooth projective variety $(X, H)$. 
Often it does not give a sharp bound. 
Indeed, it is easy to obtain 
stronger inequalities on del Pezzo surfaces or K3 surfaces, 
simply by using Serre duality and the Riemann-Roch theorem. 
Finding a sharp bound is the same problem 
as the classification of the Chern characters of slope semistable sheaves. 
Such a study goes back to the work of Drezet--Le Potier \cite{dlp85} on the projective plane, 
and is recently developing in relations with Bridgeland stability conditions, 
see e.g., \cite{flz21, lr21}. 

However, only a few results are known 
for general type surfaces or higher dimensional varieties. 
In this paper, we prove several results in this direction: 

\begin{thm}[Theorem \ref{thm:bgP3}, Corollary \ref{cor:bgP3}, Theorem \ref{thm:P3positive}] 
\label{thm:P3intro}
Let $k=\overline{k}$ be an algebraically closed field of arbitrary characteristic. 
Let $E$ be a slope semistable torsion free sheaf on $\bP^3_k$ 
with slope $\mu(E) \in [0, 1]$. 
Then the following inequality holds: 
\[
\frac{\ch_2(E)}{\ch_0(E)} \leq \Theta\left( \mu(E) \right), 
\]
where the function 
$\Theta \colon [0, 1] \to \bR$ 
is defined as follows: 
\[
\Theta(t):=
\begin{cases}
-t/4 & (t \in [0, 1/2]) \\
5t/4-3/4 & (t \in (1/2, 1]). 
\end{cases}
\]
\end{thm}

\begin{thm}[Theorems \ref{thm:bggeneral}, Theorem \ref{thm:hyperpositive}] 
\label{thm:hyperintro}
Let $k=\overline{k}$ be an algebraically closed field 
of characteristic zero. 
Let $S^n_d \subset \bP^{n+1}_k$ be a smooth hypersurface 
of degree $d \geq 1$, dimension $n \geq 2$, 
$H$ the restriction of the hyperplane class on $\bP^{n+1}_k$ to $S^n_d$. 
Let $E$ be a slope semistable sheaf 
with slope $\mu_H(E) \in [0, 1]$. 
Then we have the inequality 
\[
\frac{H^{n-2}\ch_2(E)}{H^n\ch_0(E)} \leq 
\Xi(\mu_H(E)), 
\]
where we define the function $\Xi \colon [0, 1] \to \bR$ as 
\[
\Xi(t):=
\begin{cases}
\frac{1}{3}t^2-\frac{1}{12}t & (t \in [0, 1/2]) \\
\frac{1}{3}t^2+\frac{5}{12}t-\frac{1}{4 } & (t \in [1/2, 1]). 
\end{cases}
\]

Moreover, when $n=2$, 
the result also holds in positive characteristic. 
\end{thm}

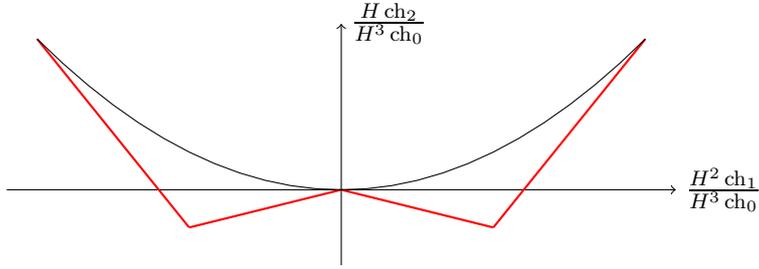
\begin{figure}[htb]
\begin{center}
\begin{tikzpicture}[scale=4]
\draw[->] (-1.1, 0) -- (1.1, 0) node[right]{$\frac{H^2\ch_1}{H^3\ch_0}$}; 
\draw[->] (0, -1/4) -- (0, 0.55) node[right]{$\frac{H\ch_2}{H^3\ch_0}$}; 
\draw[red, thick, domain=-1:-1/2] plot(\x, -5/4*\x-3/4);
\draw[red, thick, domain=-1/2:-0] plot(\x, 1/4*\x);
\draw[red, thick, domain=0:1/2] plot(\x, -1/4*\x); 
\draw[red, thick, domain=1/2:1] plot(\x,  5/4*\x-3/4); 
\draw[domain=-1:1] plot(\x, 1/2*\x*\x); 
\end{tikzpicture}
\end{center}
\caption{strong BG inequality on $\bP^3$.}
\end{figure}

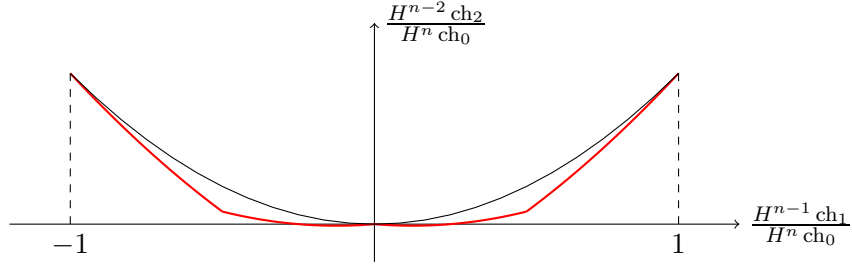
\begin{figure}[htb]
\begin{center}
\begin{tikzpicture}[scale=4]
\draw[->] (-1.2, 0) -- (1.2, 0) node[right]{$\frac{H^{n-1}\ch_1}{H^n\ch_0}$}; 
\draw[->] (0, -1/8) -- (0, 2/3) node[right]{$\frac{H^{n-2}\ch_2}{H^n\ch_0}$}; 
\draw[red, thick, domain=-1:-1/2] plot(\x, 1/3*\x*\x-5/12*\x-1/4);
\draw[red, thick, domain=-1/2:-0] plot(\x, 1/3*\x*\x+1/12*\x);
\draw[red, thick, domain=0:1/2] plot(\x, 1/3*\x*\x-1/12*\x); 
\draw[red, thick, domain=1/2:1] plot(\x,  1/3*\x*\x+5/12*\x-1/4); 
\draw[domain=-1:1] plot(\x, 1/2*\x*\x); 
\draw[dashed] (1, 0)  node[below]{$1$} -- (1, 1/2); 
\draw[dashed] (-1, 0)  node[below]{$-1$} -- (-1, 1/2); 
\end{tikzpicture}
\end{center}
\caption{strong BG inequality on hypersurfaces.}
\end{figure}

\subsection{Idea of proof} 
The key ingredient of the proofs is 
the theory of {\it tilt-stability} 
in the derived category 
(see Section \ref{sec:pre} 
for the definition and basic properties). 
In particular, we use 
(1) {\it the restriction technique for tilt-stability}, 
and (2) {\it the generalized BG type inequality} on $\bP^3$. 

(1) The restriction result 
for tilt-stability (Lemma \ref{lem:restriction}), 
first found by Feyzbakhsh \cite{fey16} and Li \cite{li19b}, 
enables us to reduce the problems to the surface case. 
In contrast to the usual effective restriction theorem 
for slope semistability (see e.g. \cite{lan10}), 
we are often allowed to cut by a hyperplane of degree one, 
not its higher multiples. 
By this observation, 
we are able to deduce Theorem \ref{thm:hyperintro} for $S^n_d$ 
from the case of the surface $S^2_d$ of the same degree. 
Similarly, we obtain the result on $\bP^3$ 
by restricting to the low dgree hypersurfaces, 
but to obtain the strong result as in Theorem \ref{thm:P3intro}, 
we use the quadric and quartic surfaces, 
not only $\bP^2$. 

(2) The generalized BG type inequality is 
the inequality for the Chern characters 
of tilt-semistable objects on $D^b(\bP^3)$, 
involving the third part of the Chern character, 
which depends on the parameter $(\beta, \alpha) \in \bR^2$ 
of tilt-stability conditons 
(see Theorem \ref{thm:bg3P3}). 
Let us call this inequality as $BG(\beta, \alpha)$. 

For a slope semistable sheaf $E \in \Coh(S^2_d)$, 
it is known that the torsion sheaf $\iota_*E \in \Coh(\bP^3)$ 
is tilt-semistable when the parameter $\alpha$ is sufficiently large. 
The next important step is to analyze 
the tilt-stability of $\iota_*E$ 
when we decrease the parameter $\alpha$, 
since  if it remains tilt-semistable for the smaller $\alpha$, 
we get the better inequality $BG(\beta, \alpha)$. 
To get the better bound on $\alpha$, 
we use Theorem \ref{thm:P3intro} in a crucial way.

\subsection{Complete intersections}
One may wish to generalize Theorem \ref{thm:hyperintro}
to all complete intersections in the projective space. 
For this, we need the conjectural $\ch_3$-inequality $BG(\beta, \alpha)$
on threefolds $S^3_d \subset \bP^4$ for $d \geq 2$. 
At this moment, the conjecture is solved only when $d \leq 5$ 
(cf. \cite{li19b, li19a}).

\subsection{Relation with the existing works}
There are several works investigating strong forms of the BG inequality. 
In \cite{har78}, Hartshorne obtains a sharp bound 
for possible Chern characters of semistable vector bundle 
of rank two on $\bP^3$. 
In \cite{li19a}, Li proves the stronger BG inequality 
for all Fano threefolds of Picard rank one, in particular for $\bP^3$. 
Our Theorem \ref{thm:P3intro} is stronger than that. 

In \cite{ss18}, Schmidt-Sung obtain a sharp bound 
for rank two vector bundles on hypersurfaces in $\bP^3$ 
with a similar method as in this paper. 
Theorem \ref{thm:hyperintro} for higher rank bundles is completely new. 
See also \cite{ms11, ms18, tod14c} 
for other results concerning rank two vector bundles. 

In \cite{kos20, li19b}, stronger BG inequalities on 
certain classes of Calabi-Yau threefolds are studied, 
with a different argument using the restriction to complete intersection curves. 
The advantage of our approach is the fact that 
we can uniformly treat hypersurfaces of all degrees 
at the same time, with short and simple computations. 

See also \cite{bay18, bl17, fey19, fey20, fl18, ft19, ft20} 
for other interesting applications of the wall-crossing in tilt-stability.

\subsection{Plan of the paper}
The paper is organized as follows. 
Until Section \ref{sec:hyper}, we work over 
an algebraically closed field of characteristic zero. 
In Section \ref{sec:pre}, 
we recall basic notions in the theory of tilt-stability. 
In Section \ref{sec:wc}, we summarize results 
obtained via wall-crossing arguments in tilt-stability. 
In Section \ref{sec:P3}, 
we prove Theorem \ref{thm:P3intro}. 
In Section \ref{sec:hyper}, we prove Theorem \ref{thm:hyperintro}. 
Finally in Section \ref{sec:positive}, 
we discuss the case of positive characteristic. 

\begin{ACK}
The author would like to thank Professor Arend Bayer
for useful discussions and comments. 
The author was supported by 
ERC Consolidator grant WallCrossAG, no.~819864. 
\end{ACK}

\begin{NaC}
Until Section \ref{sec:hyper}, we work over 
an algebraically closed field of characteristic zero, 
while in Section \ref{sec:positive} we work in positive characteristic. 
We use the following notations: 
\begin{itemize}
\item $\Coh(X)$: the category of coherent sheaves on a variety $X$. 
\item $D^b(X):=D^b(\Coh(X))$: the bounded derived category of coherent sheaves. 
\item $\ch^\beta=(\ch^\beta_0, \ch^\beta_1, \cdots, \ch^\beta_n):=e^{-\beta H}.\ch$: 
the $\beta$-twisted Chern character for a real number $\beta \in \bR$ and an ample divisor $H$. 
\end{itemize}
\end{NaC}

\section{Prelimilaries} \label{sec:pre}
Until the end of Section \ref{sec:hyper}, 
we work over an algebraically closed field of characteristic zero. 
In this section, we quickly recall 
the notion of tilt-stability on the derived categories 
and its properties. 
See \cite{bms16, bmt14a, li19b, ms17} for the details. 
Let $X$ be a smooth projective variety of dimension $n \geq 2$, 
and $H$ an ample divisor on $X$. 
Let us fix real numbers $\beta, \alpha \in \bR$ with 
$\alpha > \beta^2/2$.

\subsection{Definition of tilt-stability}
Let us first define a slope function on the category $\Coh(X)$ as 
\[
\mu_H:=\frac{H^{n-1}\ch_1}{H^n\ch_0} \colon \Coh(X) \to \bR \cup \{+\infty\}. 
\]
We define the notion of {\it $\mu_H$-stability} 
(or {\it slope stability}) for coherent sheaves 
in the usual way. 
We are then able to construct 
a new heart in the derived category $D^b(X)$ 
using the notion of {\it torsion pair} and {\it tilting} 
(cf \cite{hrs96}). 
Let us define full subcategories 
$\mcT_\beta, \mcF_\beta \subset \Coh(X)$ as follows: 
\begin{align*}
&\mcT_\beta:=\left\langle 
	T \in \Coh(X) \colon T \mbox{ is } \mu_H \mbox{-semistable with } \mu_H(T) >\beta 
	\right\rangle, \\
&\mcF_\beta:=\left\langle 
	F \in \Coh(X) \colon F \mbox{ is } \mu_H \mbox{-semistable with } \mu_H(F) \leq \beta 
	\right\rangle. 
\end{align*}
Here, for a set of objects $S \subset \Coh(X)$, 
we denote by $\langle S \rangle \subset \Coh(X)$ 
the extension closure of $S$. 
By the existence of Harder-Narasimhan filtrations 
with respect to slope stability, 
the pair $(\mcT_\beta, \mcF_\beta)$ is a torsion pair 
on $\Coh(X)$. 
Hence the category 
\[
\Coh^\beta(X):=\left\langle 
	\mcF_\beta[1], \mcT_\beta
	\right\rangle 
\subset D^b(X), 
\]
defined as the extension closure of 
$\mcF_\beta[1] \cup \mcT_\beta$ in $D^b(X)$, 
is the heart of a bounded t-structure on $D^b(X)$. 
We now define a new slope function on $\Coh^\beta(X)$ as 
\[
\nu_{\beta, \alpha}:=\frac{H^{n-2}\ch_2-\alpha H^n\ch_0}{H^{n-1}\ch_1-\beta H^n\ch_0} 
\colon \Coh^\beta(X) \to \bR \cup \{+\infty\}, 
\]
and the notion of {\it $\nu_{\beta, \alpha}$-stability} (or {\it tilt-stability}) 
for objects in $\Coh^\beta(X)$ as similar to $\mu_H$-stability. 
See for example \cite[Section 2]{li19b} 
for basic properties of tilt-stability.

\subsection{Bogomolov-Gieseker type inequalities}
Here we recall several variants of 
Bogomolov-Gieseker (BG) type inequalities. 
For an object $E \in D^b(X)$, we define 
\begin{align*}
&\Delta(E):=\left(\ch_1(E) \right)^2-2\ch_0(E)\ch_2(E), \\
&\overline{\Delta}_H(E):=\left(H^{n-1}\ch_1(E) \right)^2-2H^n\ch_0(E)H^{n-2}\ch_2(E). 
\end{align*}
The following is the classical BG inequality: 
\begin{thm}[\cite{bog78, gie79, lan04}] \label{thm:bgusual}
Every $\mu_H$-semistable torsion free sheave $E$ satisfies the inequality 
\[
\overline{\Delta}_H(E) \geq H^{n-2}\Delta(E) \geq 0. 
\]
\end{thm}

It is known that tilt-stable objects also satisfy the BG inequality: 
\begin{thm}[{\cite[Theorem 3.5]{bms16}}] \label{thm:bgtilt}
Every tilt-semistable object $E \in D^b(X)$ 
satisfies the inequality 
\[
\overline{\Delta}_H(E) \geq 0. 
\]
\end{thm}

The following generalized BG type inequality on $\bP^3$ 
plays a crucial role in this paper: 
\begin{thm}[\cite{bmt14a, mac14}] \label{thm:bg3P3}
Let $\alpha, \beta \in \bR$ be real numbers with 
$\alpha > \beta^2/2$. 
For every $\nu_{\beta, \alpha}$-semistable object $E \in \Coh^\beta(\bP^3)$, 
we have the inequality 
\[
\left(2\alpha-\beta^2 \right)\overline{\Delta}_H(E)
+4\left(H\ch_2^\beta(E) \right)^2
-6H^2\ch_1^\beta(E)\ch_3^\beta(E)
\geq 0. 
\]
\end{thm}

\section{Wall-crossing arguments} \label{sec:wc}
In this section, we summarize various wall-crossing arguments in tilt-stability, 
developed in \cite{bms16, fey16, li19a, kos20}, and others. 
As in the previous section, we denote by $X$ 
a smooth projective variety of dimension $n \geq 2$, 
and by $H$ an ample divisor on $X$. 

We put 
$\mcS:=\left\{(\beta, \alpha) \colon \alpha > \beta^2/2 \right\} \subset \bR^2$, 
and call it as {\it a space of tilt-stability conditions}. 
Recall that a {\it wall} for an object $E \in D^b(X)$ 
with respect to tilt-stability 
is defined as a connected component of solutions $(\beta, \alpha) \in \mcS$ 
of an equation $\nu_{\beta, \alpha}(F)=\nu_{\beta, \alpha}(E)$ 
for an inclusion $F \subset E$ in the tilted heart $\Coh^\beta(X)$, 
where $F \in \Coh^\beta(X)$ is a tilt-semistable object.  
It is easy to see that a wall $W$ for $E$ is a line segment, 
and satisfies one of the following two properties: 
\begin{enumerate}
\item $W$ passes through the point $p_H(E)$ 
when $\ch_0(E) \neq0$, 
\item $W$ has a fixed slope $H^{n-2}\ch_2(E)/H^{n-1}\ch_1(E)$ 
when $\ch_0(E)=0$. 
\end{enumerate}
Here, for an object $E \in D^b(X)$ with $\ch_0(E) \neq 0$, 
we define the point $p_H(E) \in \bR^2$ as follows: 
\[
p_H(E):=\left(
	\frac{H^{n-1}\ch_1(E)}{H^n\ch_0(E)}, 
	\frac{H^{n-2}\ch_2(E)}{H^n\ch_0(E)}
	\right). 
\]

\subsection{Restriction lemma for tilt-stability}
First let us recall Feyzbakhsh's restriction lemma from \cite{fey16} 
(see also \cite[Lemma 5.1]{li19b}): 
\begin{lem}[{\cite[Corollary 4.3]{fey16}}] \label{lem:feyrestr}
Let $d \geq 1$ be a positive integer, 
$\alpha>0$ a positive real number. 
Let $E \in \Coh^0(X)$ be a slope stable reflexive sheaf. 
Suppose that the following conditions hold: 
\begin{itemize}
\item we have $E(-dH)[1] \in \Coh^0(X)$,  
\item the objects $E, E(-dH)[1] \in \Coh^0(X)$ are $\nu_{0, \alpha}$-stable, 
\item we have the equality $\nu_{0, \alpha}(E)=\nu_{0, \alpha}\left(E(-dH)[1] \right)$. 
\end{itemize}
Then for any irreducible hypersurface $Y_d \in |dH|$, 
the restriction $E|_{Y_d}$ is $\mu_{H_{Y_d}}$-stable. 
\end{lem}

We use the following terminology: 
\begin{defin}
Fix an integer $d \geq 0$. 
We say that a function $f \colon [0, 1] \to \bR$ 
is {\it star-shaped along the line $\beta=d$} 
if the following condition holds: 
for every real number $t \in [0, 1]$, 
the line segment connecting the point 
$(t, f(t))$ and the point $(d, d^2/2)$ 
is above the graph of $f$. 
\end{defin}

We will use the following variant of \cite[Proposition 5.2]{li19b}: 

\begin{lem}[cf. {\cite[Proposition 5.2]{li19b}}] \label{lem:restriction}
Let $d \geq 1$ be an integer, 
and $f \colon [0, 1] \to \bR$ be a star-shaped function 
along the lines $\beta=0, d$ 
with $f(0)=0, f(1)=1/2$, satisfying 
\[
t^2-\frac{d}{2}t \leq f(t) \leq \frac{1}{2}t^2
\]
for every $t \in [0, 1]$. 
Assume that there exist objects 
$E' \in D^b(X)$ satisfying the following conditions: 
\begin{enumerate}
\item [(a)] $E'$ is either $\nu_{0, \alpha}$-semistable for some $\alpha>0$, 
or $\nu_{d, \alpha'}$-semistable for some $\alpha' > d^2/2$. 

\item[(b)] $\mu_H(E') \in [0, 1], \quad 
\frac{H^{n-2}\ch_2(E')}{H^{n-1}\ch_0(E')} > f\left(\mu_H(E') \right)$. 
\end{enumerate}
Then we can choose such an object $E$ so that 
the restriction $E|_{Y_d}$ of $E$ 
to an irreducible hypersurface 
$Y_d \in |dH|$ is $\mu_{H_{Y_d}}$-semistable. 
\end{lem}
\begin{proof}
By Theorem \ref{thm:bgtilt}, 
every tilt-semistable object $E$ 
satisfies the inequality $\overline{\Delta}_H(E) \geq 0$. 
Hence we may choose an object $E$ 
which has the minimum discriminant $\overline{\Delta}_H$ among 
those satisfying the conditions (a) and (b). 
We claim that such an object $E$ is 
$\nu_{0, \alpha}$-stable for all $\alpha>0$, 
and $\nu_{d, \alpha'}$-stable for all $\alpha' > d^2/2$. 
Assume for a contradiction that 
there is a wall for $E$ along the line $\beta=0$. 
Then there exists a Jordan-H{\"o}lder factor $F$ of $E$ 
such that the point $p_H(F)$ lies on the line segment 
connecting $p_H(E)$ and $(0, \alpha)$ for some $\alpha >0$. 
As we assume the function $f$ is star-shaped along 
the line $\beta=0$, 
the object $F$ also satisfies the conditions (a) and (b). 
Moreover, we have $\overline{\Delta}_H(F) < \overline{\Delta}_H(E)$ 
by \cite[Corollary 3.10]{bms16}, 
which contradicts the minimality assumption of $\overline{\Delta}(E)$. 
Similarly, we can see that the object $E$ cannot be destabilized 
along the line $\beta=d$ or the vertical wall $\beta=\mu_H(E)$. 

Hence the objects $E, E(-dH)[1] \in \Coh^0(X)$ 
are $\nu_{0, \alpha}$-stable for all $\alpha >0$. 
In particular, the $\nu_{0, \alpha}$-stability of $E$ for $\alpha \gg 0$ 
implies that $E$ is a coherent sheaf. 
Moreover, as in the first paragraph of the proof of \cite[Lemma 5.1]{li19b}, 
the $\nu_{0, \alpha}$-stability of $E(-dH)[1]$ implies that $E$ is reflexive. 
Note also that, by the assumption 
$t^2-\frac{d}{2}t \leq f(t)$, the line passing through 
the points $p_H(E)$ and $p_H(E(-dH))[1]$ 
intersects with the $\alpha$-axis at $(0, \alpha_0)$ 
for some positive real number $\alpha_0$. 
Hence the assumptions of Lemma \ref{lem:feyrestr} are satisfied, 
and we can conclude that 
the restriction $E|_{Y_d}$ is slope semistable. 
\end{proof}

\subsection{Strong BG inequalities for tilt-stabile objects}
Let $\mcD$ be a set of objects $E \in \Coh^0(X)$ 
satisfying one of the following conditions: 
\begin{itemize}
\item $E \in \Coh(X)$ and it is $\mu_H$-semistable with $\ch_0(E) >0$, 
\item $\mcH^{-1}(E)$ is $\mu_H$-semistable and $\dim \mcH^0(E) \leq n-2$. 
\end{itemize}
We will also use the following lemma: 
\begin{prop}[{\cite[Proposition 2.5]{kos20}}] \label{prop:tiltbg} 
Let $f \colon [0, 1] \to \bR$ be a star-shaped function along the line $\beta=0$. 
Assume that for every object $E \in \mcD$, the inequality 
\[
\frac{H^{n-2}\ch_2(E)}{H^{n}\ch_0(E)} 
\leq f\left(\mu_H(E) \right) 
\]
holds. Then for every $\alpha >0$ and 
$\nu_{0, \alpha}$-semistable object $E$ with $\ch_0(E) \neq 0$, 
the same inequality holds. 
\end{prop}

\subsection{Bounding first walls for torsion sheaves}
In the proof of Theorem \ref{thm:hyperintro}, 
it is important to bound the first possible wall for 
sheaves supported on divisors. 
The following lemma is a useful general fact: 
\begin{lem}[{\cite[Lemma 3.6]{kos20}}, \cite{li19b}] \label{lem:firstwall}
Let $Y_d \in |dH|$ be a smooth hypersurface of degree $d \geq 1$, 
denote by $\iota \colon Y_d \hookrightarrow X$ the embedding. 
Let $E \in \Coh(Y_d)$ be a $\mu_{H_{Y_d}}$-semistable torsion free sheaf. 

Assume that there exists a wall 
for $\iota_*E \in \Coh^0(X)$ with respect to tilt-stability 
with end points $(\beta_1, \alpha_1), (\beta_2, \alpha_2)$ satisfying 
$\beta_1 < 0 < \beta_2$. 
Then we have $\beta_2-\beta_1 \leq d$. 
\end{lem}
\begin{proof}
The same proof as in \cite[Lemma 3.6]{kos20} works, 
where the author considers the case of $d=6$. 
\end{proof}

\section{BG inequality on the projective space} \label{sec:P3}
In this section, we investigate the stronger form of the BG inequality 
on the three dimensional projective space. 
For a hypersurface $S_d \subset \bP^3$ of degree $d \in \bZ_{>0}$, 
we denote by $H:=H_{\bP^3}|_{S_d}$ 
the restriction of the hyperplane on $\bP^3$ to the surface $S_d$. 
The following two lemmas are well-known: 

\begin{lem}[cf. \cite{rud94}] \label{lem:bgS2}
Let $S_2 \subset \bP^3$ be a smooth quadric hypersurface. 
Let $E \in \Coh(S_2)$ be a torsion free $\mu_H$-semistable sheaf 
with slope $\mu_H(E) \in [0, 1]$. 
Then the inequality 
\[
\frac{\ch_2(E)}{H^2\ch_0(E)} \leq \Gamma\left( \mu_H(E) \right)
\]
holds, where we define the function 
$\Gamma \colon [0, 1] \to \bR$ as follows: 
\[
\Gamma(t):=
\begin{cases}
-t/2 & (t \in [0, 1/2)) \\
0 & (t=1/2) \\
3t/2-1 & (t \in (1/2, 1]). 
\end{cases}
\]
\end{lem}
\begin{proof}
For $i=1, 2$, let $h_i$ be divisors on $S_2 \cong \bP^1 \times \bP^1$ 
such that $\mcO_{S_2}(h_1)=\mcO_{S_2}(1, 0)$, 
$\mcO_{S_2}(h_2)=\mcO_{S_2}(0, 1)$. 
Note that we have 
$\mu_H(\mcO_{S_2}(h_i))=1/2$ and 
$\ch_2(\mcO_{S_2}(h_i))=0$. 

Let $E \in \Coh(S_2)$ be a slope stable vector bundle 
with $\mu_H(E) \in [0, 1/2]$, not isomorphic to 
$\mcO_{S_2}(h_1)$ nor $\mcO_{S_2}(h_2)$. 
By stability of $E$ and Serre duality, we have 
\[
\hom\left( \mcO_{S_2}(h_i), E \right)=0=\ext^2\left(\mcO_{S_2}(h_i), E \right). 
\]
Hence by the Riemann-Roch theorem, we have 
\begin{align*}
0 \geq -\ext^1\left(\mcO_{S_2}(h_i), E \right)
&=\chi\left(\mcO_{S_2}(h_i), E \right) \\
&=\int_{S_2}\ch(E).(1, H-h_i, 0) \\
&=\ch_2(E)+(H-h_i)\ch_1(E). 
\end{align*}
Summing up these inequalities for $i=1, 2$, we get 
\begin{equation} \label{eq:mu1/2}
2\ch_2(E) \leq -\left(2H-(h_1+h_2) \right)\ch_1(E)
=-H\ch_1(E)
\end{equation}
as required. 
When $\mu_H(E) \in [1/2, 1]$, 
we get the required inequality 
by applying the inequality (\ref{eq:mu1/2}) 
to the bundle $E^\vee(H)$. 
\end{proof}

\begin{lem} \label{lem:bgS4}
Let $S_4 \subset \bP^3$ be a smooth quartic hypersurface. 
Let $E \in \Coh(S_4)$ be a torsion free $\mu_H$-semistable sheaf 
with slope $\mu_H(E)=1/2$. 
Then the inequality 
\[
\frac{\ch_2(E)}{H^2\ch_0(E)} \leq 
-\frac{1}{8}. 
\]
holds. 
\end{lem}
\begin{proof}
We may assume that $E$ is a slope stable vector bundle 
with $\ch_1(E)=\ch_0(E)H/2$. 
Recall that we have $\ch(E) \in H^*_{alg}(X, \bZ)$, 
so we have $\ch_0(E)=2a$ for some positive integer $a \in \bZ_{>0}$ 
and $\ch_2(E) \in \bZ$. 
First we claim that the bundle $E$ is non-spherical. 
If otherwise, we have 
\[
2=\chi(E, E)
=2\ch_0(E)^2-\ch_0(E)^2+2\ch_0(E)\ch_2(E)
\]
and hence 
\[
\ch_2(E)=\frac{1}{\ch_0(E)}-\frac{\ch_0(E)}{2}
=\frac{1}{2a}-a 
\notin \bZ, 
\]
which is a contradiction. 
Now for a non-spherical stable bundle $E$, 
we have the inequality 
$0 \geq \chi(E, E)$, 
from which we deduce the required inequality. 
\end{proof}

Let us define a periodic function 
$\gamma \colon \bR \to \bR$ 
with $\gamma(t+1)=\gamma(t)$ as follows: 
\[
\gamma(t):=
\begin{cases}
\frac{1}{2}t^2+\frac{1}{4}t & (t \in [0, 1/2]), \\
\frac{1}{2}t^2-\frac{5}{4}t+\frac{3}{4} & (t \in [1/2, 1]). 
\end{cases}
\]
We then define a function $\Theta \colon \bR \to \bR$ as 
$\Theta(t):=t^2/2-\gamma(t)$. 

Now we are ready to prove 
the following stronger BG inequality on $\bP^3$:

\begin{thm} \label{thm:bgP3}
Let $E \in \Coh^0(\bP^3)$ be a $\nu_{0, \alpha}$-semistable object 
for some positive real number $\alpha >0$, 
with slope $\mu_H(E) \in [0, 1]$. 
Then the inequality 
\begin{equation} \label{eq:bgP3}
\frac{H\ch_2(E)}{H^3\ch_0(E)} \leq \Theta\left( \mu_H(E) \right)
\end{equation}
holds. 
\end{thm}
\begin{proof}
Assume for a contradiction that 
there exists a $\nu_{0, \alpha}$-semistable object $E$ 
with $\mu_H(E) \in [0, 1]$ violating the inequality (\ref{eq:bgP3}). 
Note that the function 
$\Theta \colon [0, 1] \to \bR$ 
is star-shaped along the lines $\beta=0, 2, 4$, 
and satisfies the inequality 
$t^2-dt/2 \leq \Theta(t)$ 
for all $t \in [0, 1]$ and $d \geq 2$. 
Hence by Lemma \ref{lem:restriction}, 
we may assume that the restrictions $E|_{S_d}$ are 
slope semistable vector bundle on general hypersurfaces 
$S_d \subset \bP^3$ of degree $d=2, 4$. 
Note that we have $\mu_H(E|_{S_d})=\mu_H(E)$. 
When $\mu_H(E) \neq 1/2$ (resp. $\mu_H(E)=1/2$), 
we get a contradiction by Lemma \ref{lem:bgS2} 
(resp. Lemma \ref{lem:bgS4}). 
\end{proof}

\begin{cor} \label{cor:bgP3}
Every  torsion free slope semistable sheaf 
$E \in \Coh(\bP^3)$ satisfies the inequality 
\[
\frac{H\ch_2(E)}{H^3\ch_0(E)} \leq 
\Theta\left( \mu_H(E) \right). 
\]

In particular, we have a continuous family of tilt-stability 
parametrized by pairs $(\beta, \alpha)$ of real numbers satisfying 
$\alpha > \Theta(\beta)$. 
\end{cor}
\begin{proof}
By definition, the function $\Theta$ satisfies 
$\Theta(t+1)=\Theta(t)+t+1/2$ for all $t \in \bR$. 
Hence it is enough to prove 
the assertion for semistable sheaves $E$ 
with slope $\mu_H(E) \in [0, 1]$, 
which directly follows from Theorem \ref{thm:bgP3}. 
\end{proof}

\section{BG inequality on hypersurfaces} \label{sec:hyper}
The goal of this section is to prove the following theorem: 

\begin{thm} \label{thm:bggeneral}
Let $S^n_d \subset \bP^{n+1}$ be a smooth hypersurface 
of degree $d \geq 1$, dimension $n \geq 2$. 
Let $E$ be a $\nu_{0, \alpha}$-semistable object 
for some $\alpha >0$, 
with slope $\mu_H(E) \in [0, 1]$. 
Then we have the inequality 
\begin{equation} \label{eq:xi}
\frac{H^{n-2}\ch_2(E)}{H^n\ch_0(E)} \leq 
\Xi(\mu_H(E)), 
\end{equation}
where we define the function $\Xi \colon [0, 1] \to \bR$ as 
\[
\Xi(t):=
\begin{cases}
\frac{1}{3}t^2-\frac{1}{12}t & (t \in [0, 1/2]) \\
\frac{1}{3}t^2+\frac{5}{12}t-\frac{1}{4 } & (t \in [1/2, 1]). 
\end{cases}
\]
\end{thm}

First we reduce the problem to the case of surfaces: 
\begin{lem} \label{lem:cut}
Assume that every slope stable sheaf $E$ on surfaces 
$S^2_d$ with slope $\mu_H(E) \in [0, 1/2]$ 
satisfies the inequality (\ref{eq:xi}). 
Then Theorem \ref{thm:bggeneral} holds. 
\end{lem}
\begin{proof}
First note that the function $\Xi$ is star-shaped along the lines $\beta=0, 1$, 
and satisfies $t^2-t/2 \leq \Xi(t)$. 
Hence by Lemma \ref{lem:restriction} 
and induction on $n \geq 2$, 
it is enough to prove the assertion for $n=2$. 

As we assume that the inequality (\ref{eq:xi}) holds 
for every slope semistable sheaf $E \in \Coh(S^2_d)$ 
with $\mu_H(E) \in [0, 1/2]$, it also holds 
when $\mu_H(E) \in [1/2, 1]$, 
by applying the assumed inequality to $E^\vee(H)$. 
Now the same inequality (\ref{eq:xi}) holds 
for every $\nu_{0, \alpha}$-semistable objects 
by Proposition \ref{prop:tiltbg}. 
\end{proof}

In the following, we fix a smooth hypersurface 
$S^2_d \subset \bP^3$ of degree $d \geq 1$, 
and denote by 
$\iota \colon S^2_d \hookrightarrow \bP^3$ 
the embedding.

\begin{lem} \label{lem:chern}
Let $E \in D^b(S^2_d)$ be an object and put 
$r:=\ch_0(E), 
a:=H\ch_1(E)/d, 
b:=\ch_2(E)$. 
Then we have 
\begin{equation} \label{eq:ch*}
\ch(\iota_*E)=\left( 
0, drH, \left(a-\frac{d}{2}r \right)dH^2, 
b-\frac{d^2}{2}a+\frac{d^3}{6}r
	\right). 
\end{equation}
\end{lem}
\begin{proof}
Using Grothendieck-Riemann-Roch theorem for embeddings, 
we have 
\[
\iota_*\left((r, \ch_1(E), b).\td_{S^2_d} \right)
=\ch(\iota_*E)\td_{\bP^3}. 
\]
Combining with the facts 
\begin{equation} \label{eq:tdhyper}
\td_{S^2_d}=\left(1, \left(2-\frac{d}{2} \right)H_S, \frac{d^3}{6}-d^2+\frac{11}{6}d \right), \quad 
\td_{\bP^3}=\left(1, 2H, \frac{11}{6}H^2, 1 \right), 
\end{equation}
the straightforward computation yields the result. 
\end{proof}

\begin{lem} \label{lem:wall}
Let $E \in \Coh(S^2_d)$ be a slope semistable vector bundle on $S^2_d$ 
with slope $\mu:=\mu_H(E) \in [0, 1/2]$. 
Then the sheaf $\iota_*E \in \Coh^0(\bP^3)$ is 
$\nu_{0, \alpha}$-semistable for all positive 
$\alpha \geq \alpha_\mu$, 
where the real number $\alpha_\mu$ is defined as follows: 
\begin{equation} \label{eq:alpha}
\alpha_\mu:=
-\mu^2+\frac{2d-1}{4}\mu. 
\end{equation}
\end{lem}
\begin{proof}
Let $W$ be a wall for $\iota_*E$ 
with respect to $\nu_{\beta, \alpha}$-stability. 
Note that the wall $W$ is the line segment 
with slope $\mu-d/2$. 
Note also that by Corollary \ref{cor:bgP3}, 
the end points of the wall $W$ 
are on the graph of $\Theta$. 
Let $\beta_1 < 0 < \beta_2$ be their $\beta$-coordinates. 

By Lemma \ref{lem:firstwall}, 
we have $\beta_2-\beta_1 \leq d$ 
and hence the slope of the line 
passing through the points 
$\left(\beta_2, \Theta(\beta_2) \right)$, 
$\left(\beta_2-d, \Theta(\beta_2-d) \right)$ 
should be smaller than or equal to that of $W$, 
i.e., $\beta_2-d/2 \leq \mu-d/2$. 
Hence every wall for $\iota_*E$ 
should be below the line 
\[
y_\mu=(\mu-d/2)(x-\mu)+\Theta(\mu). 
\]
By computing $\alpha_\mu:=y_\mu(0)$, we get the result. 
\end{proof}

\begin{prop} \label{prop:bg2dim}
Let $E \in \Coh(S^2_d)$ be 
a slope semistable vector bundle on $S^2_d$ 
with slope $\mu:=\mu_H(E) \in [0, 1/2]$. 
Then the inequality (\ref{eq:xi}) holds. 
\end{prop}
\begin{proof}
By Lemma \ref{lem:wall}, 
the sheaf $\iota_*E \in \Coh^0(\bP^3)$ 
is $\nu_{0, \alpha_\mu}$-semistable. 
Hence by the generalized BG type inequality Theorem \ref{thm:bg3P3}, 
we get 
\begin{equation} \label{eq:ch3}
2\alpha_\mu\overline{\Delta}_H(\iota_*E)
+4\left(H\ch_2(\iota_*E) \right)^2
-6H^2\ch_1(\iota_*E)\ch_3(\iota_*E) \geq 0. 
\end{equation}
Let us put 
$r:=\ch_0(E), 
a:=H\ch_1(E)/d, 
b:=\ch_2(E)$.
Note that we have $H_S^2=d$ and $\mu=a/r$. 
Using the equations (\ref{eq:ch*}) and (\ref{eq:alpha}), 
and dividing the inequality (\ref{eq:ch3}) by $d^2r^2$, 
we get the inequality 
\[
-2\mu^2+\left(d-\frac{1}{2} \right)\mu+4\left(\mu-\frac{d}{2} \right)^2
-6\left(\frac{b}{dr}-\frac{d}{2}\mu+\frac{d^2}{6} \right) \geq 0,
\]
which is equivalent to the desired inequality (\ref{eq:xi}). 
\end{proof}

We are now able to finish the proof of Theorem \ref{thm:bggeneral}: 
\begin{proof}[Proof of Theorem \ref{thm:bggeneral}]
The assertion follows from Lemma \ref{lem:cut} and Proposition \ref{prop:bg2dim}. 
\end{proof}

\section{Positive characteristics} \label{sec:positive}
In this section, we work over 
an algebraically closed field $k$ of positive characteristic. 
In positive characteristic, 
the main difference with the case of characteristic zero is 
that the classical BG inequality does not hold in general. 
The failure of the BG inequality is related to the failure of Kodaira vanishing 
in positive characteristic (cf. \cite{muk13, ray78}). 

Before discussing about 
Theorems \ref{thm:P3intro} and \ref{thm:hyperintro}, 
let us recall the result of Langer \cite{lan04}. 
Let $X$ be a smooth projective variety of dimension $n \geq 2$, 
and $H$ an ample divisor on $X$. 
We say that a coherent sheaf $E \in \Coh(X)$ is 
{\it strongly $\mu_H$-semistable} if for every $e >0$, 
the sheaf $F^{e*}E$ is $\mu_H$-semistable, 
where $F \colon X \to X$ is the absolute Frobenius morphism. 
We have the following result: 

\begin{thm}[{\cite[Corollary 2.4, Theorem 3.2]{lan04}}]
Assume that we have $\mu_{H, \max}(\Omega_X) \leq 0$. 
Then every $\mu_H$-semistable torsion free sheaf $E$ is 
strongly $\mu_H$-semistable. 
Moreover, the usual BG inequality holds: 
\[
H^{n-2}\Delta(E) \geq 0. 
\]
\end{thm}

In particular, the usual BG inequality holds 
on the projective space $\bP^3$. 
Hence all the results in Section \ref{sec:P3} hold true 
also in positive characteristic, without any modification: 
\begin{thm} \label{thm:P3positive}
Theorem \ref{thm:P3intro} holds true in arbitrary characteristic. 
\end{thm}

For the arguments in Section \ref{sec:hyper}, 
the only issue in positive characteristic is Lemma \ref{lem:cut}, 
for which we used Lemma \ref{lem:restriction}. 
To apply Lemma \ref{lem:restriction} 
to hypersurfaces $S^n_d \subset \bP^{n+1}_k$, 
we need the usual BG inequality on $S^n_d$, 
which is not known in positive characteristic. 
The other arguments in Section \ref{sec:hyper} also work 
in positive characteristic, and hence: 

\begin{thm} \label{thm:hyperpositive}
When $n=2$, Theorem \ref{thm:hyperintro} holds 
in arbitrary characteristic. 
\end{thm}

\begin{cor} \label{cor:hyperpositive}
Let $k$ be an algebraically closed field of positive characteristic. 
Let $S^2_d \subset \bP^3$ be a smooth hypersurface of degree $d \geq 1$, 
denote by $H$ the restriction of the hyperplane class on $\bP^3$. 
Then for every torsion free $\mu_H$-semistable sheaf $E \in \Coh(S^2_d)$, 
we have 
\[
\overline{\Delta}_H(E) \geq 0. 
\]
\end{cor}

\begin{rmk}
In a subsequent paper \cite{kos20b}, 
we will prove the usual BG inequality 
(with respect to $\overline{\Delta}_H$) 
on $S^n_d$ for $n \geq 3$. 
The key input in \cite{kos20b} is Corollary \ref{cor:hyperpositive} above. 
Hence Theorem \ref{thm:hyperintro} will be true 
for all $S^n_d$ in arbitrary characteristic.  
\end{rmk}


\end{document}